\documentclass[12pt]{amsart} 

\usepackage{amsthm}
\usepackage{amssymb}
\usepackage{amsfonts}
\usepackage{verbatim}
 \usepackage[colorlinks,backref]{hyperref}
\newcommand{\barr}{\mathcal{B}}

\newcommand{\CC}{\mathbb{C}}
\newcommand{\RR}{\mathbb{R}}
\newcommand{\CP}{\mathbb{C}\mathbb{P}}

\newcommand{\arr}{\mathcal{A}}

\newcommand{\baseHOne}{\mathcal{B}_{H_1}}

\newcommand{\baseArr}{\mathcal{A}}

\newcommand{\baseAtInfinity}{\mathcal{B}_{\infty}}

\theoremstyle{plain}

\newtheorem{theorem}{Theorem}[section]

\newtheorem{lemma}{Lemma}[section]

\theoremstyle{definition}

\makeatletter 
\let\c@theorem=\c@thm
\let\c@lem=\c@thm
\let\c@lemma=\c@thm
\let\c@prop=\c@thm
\let\c@cor=\c@thm
\let\c@cong=\c@thm
\let\c@defn=\c@thm
\let\c@remark=\c@thm
\let\c@example=\c@thm
\let\c@note=\c@thm
\let\c@nte=\c@thm
\let\c@observe=\c@thm
\let\c@notation=\c@thm
\makeatother

\newcommand{\parLines}{\mathcal{D}}

 \usepackage{amsmath}
 \usepackage{amscd}
 \usepackage{ifthen}

 \usepackage{tikz}
 \usetikzlibrary{arrows,chains,matrix,positioning,scopes}

 \makeatletter
 \tikzset{join/.code=\tikzset{after node path={%
 \ifx\tikzchainprevious\pgfutil@empty\else(\tikzchainprevious)%
 edge[every join]#1(\tikzchaincurrent)\fi}}}
 \makeatother
 
 \tikzset{>=stealth',every on chain/.append style={join},
          every join/.style={->}}
 
 \usepgflibrary{shapes.geometric}

\begin{document}

\title[The homotopy type of unions of arrangements and pencils]{On the homotopy type of the complement of an arrangement that is a 2-generic section of the parallel connection of an arrangement and a pencil of lines}
\author{Kristopher Williams}
\address{Department of Mathematics, Doane College, Crete, NE 68333, USA}
\email{kristopher.williams@doane.edu}



\begin{abstract}  Let  \( \mathcal{A} \) be a complexified-real arrangement of lines in  \( \mathbb{C}^2. \) Let  \( H \) be any line in  \( \mathcal{A} \). Then, form a new complexified-real arrangement  \( \mathcal{B}_H = \mathcal{A} \cup \mathcal{C} \) where  \( \mathcal{C} \cup \{H\} \) is a pencil of lines with multiplicity  \( m\geq 3 \), the intersection point in the pencil is not a multiple point in  \( \mathcal{A}, \) and every line in  \( \mathcal{C} \) intersects every line in  \( \mathcal{A}\setminus \{H\} \) in points of multiplicity two. In this article, we show that for  \( H_1, H_2 \in \mathcal{A} \) we may have that  \( \mathcal{B}_{H_1} \) and  \( \mathcal{B}_{H_2} \) do not have diffeomorphic complements, but the complements of the arrangements will always be homotopy equivalent.
\end{abstract}

\maketitle

\section{Introduction}\label{sec:intro}

A hyperplane arrangement \( \arr \) is a finite collection of affine subspaces in  \( \CC^l. \) When the defining forms for all of the hyperplanes in the arrangement can be chosen to have real coefficients but the arrangement is considered as an arrangement in  \( \CC^l \), the arrangement  \( \arr \) is called a complexified real arrangement. The complement of an arrangement is denoted by $M(\arr) = \CC^l \setminus \displaystyle\cup_{H \in \arr} H$. The intersection lattice  \( L(\arr) \) of an arrangement is the set on non-empty intersections of hyperplanes in the arrangement. The intersection lattice is given the structure of a partially ordered set by reverse inclusion of the elements. All properties of the arrangement or its complement that are determined from the intersection lattice are called combinatorial.  

One of the big questions in the study of arrangements is to what extent are the combinatorics and topology of the complement related. A well known result of Orlik and Solomon \cite{Orlik-Solomon-MR558866} shows that the cohomology algebra of the complement of an arrangement is determined by the combinatorics. In \cite{Ryb98}, Rybnikov gives an example of two arrangements with the same combinatorics, but the arrangements have non-isomorphic fundamental groups (thereby, the complements cannot be homotopy equivalent). See \cite{ACCM-TopoAndCombi-MR2188450}, for more results related to this idea.  A consequence of the work of Jiang and Yau\cite{JY-Intersection-lattices-MR1631597} shows that if  the combinatorics of \( \arr_a \) and  \( \arr_b \) are different, then the complements  \( M(\arr_a) \) and  \( M(\arr_b) \) are not diffeomorphic.

In \cite{Falk-homotopy-types-MR1193601} , Falk presents an example that falls between the results of Rybnikov and Jiang and Yau. That is, Falk presents a family of pairs of arrangements  \( (\arr_1,\arr_2) \) in  \( \CC^2 \) where  \( \arr_1 \) and  \( \arr_2 \) have different combinatorics and do not have homeomorphic complements, but  \( \arr_1 \) and  \( \arr_2 \) have homotopy equivalent complements. Therefore we also have  \( \pi_1(M(\arr_1)) \cong \pi_1(M(\arr_2)). \) However, all of Falk's examples have  \(\pi_1(M(\arr_1)) \cong \mathbb{F}_p \times \mathbb{F}_q \times \mathbb{Z}_1.  \) Using the work of Fan \cite{Fan-Directproducts-MR1460414} and Eliyahu, Liberman, Schaps and Teicher \cite{ELST-Char-DSFG-converse-AG}, it was shown in \cite{KW-DPFG} that any two line arrangement complements with fundamental groups isomorphic to the same direct product of free groups will have homotopy equivalent complements even though they may have different combinatorics, thus do not have homeomorphic complements. 

The purpose of this paper is to continue the study of arrangements that have different combinatorics, complements that are not diffeomorphic, but are homotopy equivalent. The main theorems of this paper follow.

\begin{theorem}\label{arb}
Let  \( \mathcal{A} \) be a complexified-real arrangement of lines in  \( \mathbb{C}^2. \) Let  \( H \) be any line in  \( \mathcal{A} \). Then, form a new complexified-real arrangement  \( \mathcal{B}_H = \mathcal{A} \cup \mathcal{C} \) where  \( \mathcal{C} \cup \{H\} \) is a pencil of lines with multiplicity  \( m\geq 3 \), the intersection point in the pencil is not a multiple point in  \( \mathcal{A}, \) and each line in  \( \mathcal{C} \) intersects each line in  \( \mathcal{A}\setminus \{H\} \) in points of multiplicity two. For any\( H_1, H_2 \in \mathcal{A} \) we may have that  \( M(\mathcal{B}_{H_1}) \) and  \( M(\mathcal{B}_{H_2}) \)  are homotopy equivalent.
\end{theorem}
Another way to view Theorem \ref{arb} is from the perspective of parallel connections of arrangements. Let \( \arr_p \) and  \( \arr_q \) denote two central arrangements  in  \( \CC^p \) and  \( \CC^q) \) respectively. Further, let  \( H \in \arr_p \) and  \( L \in \arr_q \).  The parallel connection of the arrangements  \( P((\arr_p,H),(\arr_q,L)) \) is a central arrangement in  \( \CC^{p+q-1} \) formed by considering both arrangements as arrangements in  \( \CC^{p+q-1} \), combining the arrangements along the hyperplanes  \( H \) and  \( L, \) then placing the rest of the arrangements in general position (see \cite{Falk-Proud-ParConn-BundlesOf-Arrs-MR1877716} for a precise definition.) In \cite{Falk-Proud-ParConn-BundlesOf-Arrs-MR1877716}, it is shown that the diffeomorphism type of the complement of the parallel connection is independent of the hyperplanes chosen to form the parallel connection along. This also shows that the fundamental group of the complements of the parallel connections is independent of the hyperplanes chosen. By a well known result, the fundamental group of the parallel connection can be determined by finding the fundamental group of a 2-generic section of the complement of the arrangement in  \( \CC^2. \) However, it is unknown if the homotopy type of the sections depends on the hyperplane section chosen.

The cone over an arrangement  \( c\arr \) is formed by homogenizing the defining polynomial for a hyperplane arrangement with a new variable  \( z \) and then adding the hyperplane defined by  \( z=0 \) to the arrangement. The result is a central arrangement in one higher dimension and with one more hyperplane than the original arrangement. Let  \( c\arr \) denote the cone over the arrangement  \( \arr \) in Theorem \ref{arb} and let  \( \mathcal{P}_m \) denote a pencil of  \( m \geq 3\) lines in  \( \CC^2. \) Theorem  \ref{arb} is then equivalent to the idea that the homotopy type of a 2-generic section of the complement of the parallel connection of  \( c\arr \) and  \( \mathcal{P}_m \) is independent of the hyperplanes chosen.

We will prove Theorem \ref{arb} as a consequence of the following theorem. Essentially, Theorem \ref{main} states that one may take a complexified-real arrangement in  \( \CC^2 \) and either add  \( m-1 \) parallel lines or add a multiple point of multiplicity  \( m \) to a line in the arrangement such that in both cases the lines added are in general position with respect to the rest of the lines in the arrangement. The resulting arrangements will have homotopy equivalent complements, generally have different combinatorics and generally not have diffeomorphic complements.

\begin{theorem}\label{main}
  Let  \( \barr_{H_1} \) be a complexified real arrangement in  \( \mathbb{C}^2 \) such that \begin{itemize}
  \item  the arrangement \( \barr_{H_1} \) is the union of two subarrangements  \( \baseArr \) and  \( \mathcal{C} \) where  \( \baseArr \cap \mathcal{C}\) contains a single line we will denote by  \( H_1 \),
  \item  \( \mathcal{C} \) is a pencil of  \( m \geq 3\) lines, and
  \item  every line of  \( \mathcal{C} \setminus H_1 \) intersects each line of  \( \baseArr \) in a point of multiplicity two.
  \end{itemize} 
  Let  \( \baseAtInfinity \) be a complexified real arrangement in  \( \mathbb{C}^2 \) such that \begin{itemize}
  \item the arrangement  \( \baseAtInfinity \) is the union of two subarrangements  \( \baseArr \) and  \( \parLines \) that have no lines in common,
  \item \( \parLines \) is a set of  \( m-1 \) parallel lines
  \item every line in  \( \parLines \) intersects each line  of  \( \arr \) in a point of multiplicity two.
  \end{itemize}
  
  Then, the complement of \( \barr_{H_1} \) is homotopy equivalent to the  complement of  \( \baseAtInfinity. \)
\end{theorem}

In \cite{JY-Intersection-lattices-MR1631597}, Jiang and Yau showed that if the arrangements  \( \arr_a \) and  \( \arr_b \) have diffeomorphic complements in  \( \CP^2 \), then  \( \arr_a \) and  \( \arr_b \) have isomorphic intersection lattices.  From this result, since the coning of  \( \arr \) and the coning of  \( \barr \) are generally not lattice isomorphic, we know that the complements  \( M(\arr) \) and  \( M(\barr) \) are generally not homeomorphic.

In Section 2, some background information is given on the major tools used in the proof of the Theorem \ref{main}: cell complexes of arrangement complements, and lattice isotopies. In Section 3,  Theorem \ref{arb} and Theorem \ref{main} are proven after some lemmas related to the presentations of the fundamental groups of complements of certain arrangements. Finally, Section 4 contains an example of  Theorem \ref{main}.

\section{Background}\label{sec:background}

We follow the notation presented in \cite{OT-Arrs-MR1217488}. 
 
\subsection{Presentations and Complexes}\label{ssec:complex}
Let  \( \mathcal{P} \) be a finite group presentation. We may construct a 2-complex, denoted by  \( \vert \mathcal{P} \vert, \) by starting with one 0-cell, attaching a 1-cell for each generator, and attaching a 2-cell for each relator by following along the generators. In \cite{Falk-homotopy-types-MR1193601} the following is described:

\begin{lemma}\label{tietze}
Let  \( \mathcal{P} \) be a group presentation and  \( \vert \mathcal{P} \vert \) the associated 2-complex. The following Tietze transformations (and their inverses) may be performed on the presentation without changing the homotopy type of the 2-complex:
\begin{itemize}
\item Tietze I:  Freely reduce a relator  \( r \) or replace  \( r \) by  \( w^{-1}r^{\pm 1}w \) where  \( w \) is any word.
\item Tietze II: Delete a generator  \( a \) and a relator of the form \(aw^{-1}  \) where  \( w \) does not involve  \( a \) and replace every occurrence of  \( a \) in other relators by  \( w. \)
\item Tietze III: Replace a relator  \( r \) with  \( rs \) where  \( s \) is another relator.
\end{itemize}
\end{lemma}

We also recall there is another Tietze transformation which has the result of forming the wedge product of the complex with a copy of the 2-sphere \( S^2: \)
\begin{itemize}
\item Tietze IV: Introduce a relator that is a consequence of other relators.
\end{itemize}

In this paper, we will need to make use of the Arvola-Randell presentation for  \( \pi_1(M(\arr)).\) We recall the facts we need below and direct the reader to \cite{OT-Arrs-MR1217488} for the full details.

Let  \( \arr \) denote a complexified-real arrangement of lines in  \( \CC^2. \) We associate a graph to the real part of the arrangement in  \( \RR^2 \) by associating to each intersection point of lines a vertex. Then, to each line segment between vertices we associate an edge. We also associate an edge to each line segment that is only attached to one vertex (ie a ray). Let  \( G \) denote the set of edges, as these will be the generators in the presentation. 

 \begin{figure}[h!]
 \begin{center} 
\includegraphics[width=3in]{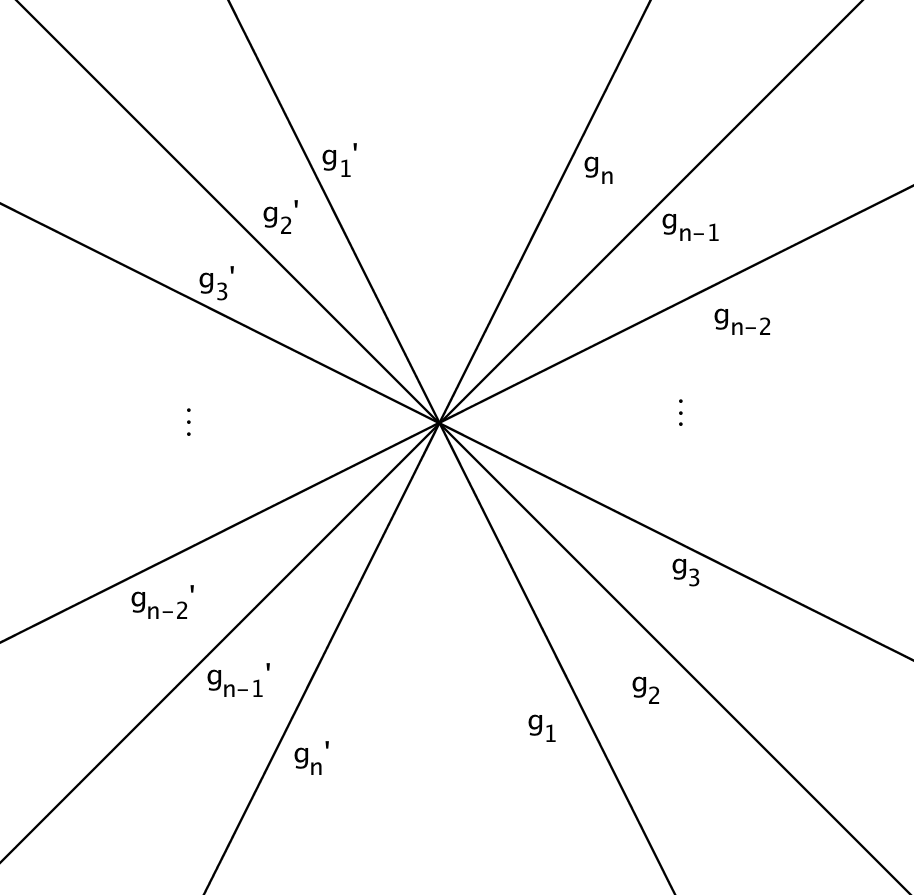}

 		\end{center}
                \caption{A typical pencil of lines.}
 											\label{pencilv}
 					\end{figure}

In Figure \ref{pencilv}, we have a typical vertex with incident edges in the real plane. At each vertex  \( v \) we have conjugation relators:
\[
g_1'g_1^{-1},   g_2' (g_2^{g_1})^{-1}, \dots, g_{n-1}'(g_{n-1}^{g_{n-2}\cdots g_1})^{-1}, g_n'g_n^{-1}
 \]

and commutation relators:
\[
[g_n, g_{n-1}\cdots g_2  g_1],
[g_n g_{n-1},g_{n-2}\cdots g_2  g_1],
\cdots
[g_n g_{n-1}\cdots g_2,  g_1]
\]
The set of commutation relators is often abbreviated as \( [g_n, g_{n-1},\dots, g_2,  g_1].\)

The Arvola-Randell presentation for  \( \pi_1(M(\arr)) \) is  \( \mathcal{P}:=\langle G \mid R \rangle \) where  \( G \) is the set of edges and  \( R \) is the set of conjugation relators and commutation relators from all of the vertices. The presentation may of course be simplified via Lemma \ref{tietze}. By the Main Theorem in \cite{Falk-homotopy-types-MR1193601}, the standard 2-complex  \( |\mathcal{P}| \) is homotopy equivalent to  \( M(\arr). \)

\subsection{The topology of  \( M(\arr) \)}

We will need two results on the topological type of the complements of arrangements. We briefly recall them here.

For the necessary background on the Arvola-Randell presentation, see \cite{Falk-homotopy-types-MR1193601} or \cite{OT-Arrs-MR1217488}. 

\begin{theorem}\label{ht}(Theorem 2.3 in \cite{Falk-homotopy-types-MR1193601}, Corollary 6.5 in \cite{CS-BraidMonodromy-MR1470093})
  Let  \( \arr \) be an arrangement in  \( \CC^2\) and let  \( \mathcal{P} \) be the Arvola-Randell presentation for  \(\pi_1( M(\arr)). \) Then the 2-complex  \( \vert \mathcal{P} \vert \) is homotopy equivalent to  \( M(\arr).\)
\end{theorem}

A smooth 1-parameter family of arrangements  \( \arr_t \) is defined by a polynomial \[
Q(\arr_t) = \prod\limits_{j=1}^n \alpha_j(z,t)
\] 
where  \[ \alpha_j(z,t) = \alpha_{j1}(t)z_1 + \alpha_{j2}(t) z_2 + \cdots + \alpha_{jl}(t) z_l  \]
and each  \( \alpha_{jm}: \mathbb{R} \to \mathbb{C} \) is a smooth function. For any value of  \( t_0 \in \mathbb{R} \), the notation  \( \arr_{t_0} \) is an arrangement.

The family is a lattice isotopy if for all values  \( t_1,t_2 \in \mathbb{R} \) the arrangements  \( \arr_{t_1} \) and  \( \arr_{t_2} \) have isomorphic lattices.

\begin{theorem}\label{diff}(Main Theorem in \cite{Randell-isotopy})
  Let  \( \arr_t \) be a lattice isotopy. Then  \( M(\arr_0) \) is diffeomorphic to  \( M(\arr_1). \)
\end{theorem}

\section{Proof of Main Results}
%
%
%

We begin with  Lemma \ref{present} that allows us to construct a complex that is homotopy equivalent to the arrangement  \( \baseHOne \) from Theorem \ref{main}.

\begin{lemma}\label{present}
  Let  \( \baseHOne \) be as described in Theorem \ref{main}. Then, \( M(\baseHOne)\) is homotopy equivalent to the complex constructed from the presentation given by 
\[
\left\langle\begin{matrix}h_1,\dots, h_n, \\  l_2,\dots, l_m \end{matrix}  ~\left\lvert~ \begin{matrix} [l_k, h_j]  & \textrm{ for } 2 \leq k \leq m\\ [h_1,l_m,\dots,l_2] & 2 \leq j \leq n \\ R_h & ~\\ \end{matrix} \right. \right\rangle
\]
  where \(\left\langle  h_1,\dots, h_{n},    \vert   R_h  \right\rangle \) is a Randell-Arvola presentation for the fundamental group of \( M(\baseArr) \) and  \( R_h \) consists of words in the generators  \( \{h_1,\dots, h_n\} \) only. 
\end{lemma}
\begin{proof}
  We begin by constructing a series of lattice isotopies to find an arrangement whose complement is diffeomorphic to the complement of  \( \baseHOne, \) but has a structure whose Randell-Arvola presentation is easier to construct.

  We choose coordinates \( \mathbb{C}[x,y] \) so that the point of intersection of the lines in  \( \mathcal{C}  \) is the origin and the line  \( H_1  \) is defined by the kernel of  \( \alpha_{H_1}=x.\) In this way, we can give a defining polynomial for  \( \mathcal{C} \) by  \( Q(\mathcal{C})=x\cdot \prod\limits_{k=2}^m(y-m_k x) \) and a defining polynomial for  \( \baseArr \) by  \( Q(\baseArr) = x \cdot \prod\limits_{j=2}^n (y-w_j x + a_j) \) where  
\begin{itemize}
\item \( m_p < m_q \) for  \( p <  q \), 
\item \( w_j \neq m_k \) for all  \( 2\leq j\leq n, 2 \leq k \leq m, \) and
\item  \( a_j\neq 0 \) for all  \( 2 \leq j \leq n. \)
\end{itemize}Thus, a defining polynomial for  \( \baseHOne \) is given by  \( Q(\baseHOne) = x\cdot \prod\limits_{k=2}^m(y-m_k x) \cdot \prod\limits_{j=2}^n (y+w_j x + a_j).\)

We can define a two parameter family of arrangements  \( \arr(t,s) \) for   \( s,t \in \mathbb{R} \) by the polynomial \[
  H(\arr(t,s))= x\cdot \prod\limits_{k=2}^m(y-m_k x - s\cdot t + (t-t^2)i) \cdot \prod\limits_{j=2}^n (y-w_j x + a_j).
  \]  Note that the arrangement  \( \baseArr \) is a subarrangement of  \( \arr(t,s) \) for all  \( s \) and  \( t. \) We define a family of subarrangements by \( \mathcal{C}(t,s) \) from the polynomial  \[Q(\mathcal{C}(t,s))=  x\prod\limits_{k=2}^m(y-m_k x - s\cdot t + (t-t^2)i).\] As  \( \baseArr \) is a finite arrangement, there exists a real number \( S>0 \) so that all multiple points of  \( \baseArr\) are inside a ball of radius  \( S \) centered at the origin. Also, we can choose an  \( R>S \) so that all points of intersection of  \( \baseArr \) and  \( \mathcal{C}(1,R) \) take place outside of ball of radius  \( R-1 \) centered at the origin. 

Therefore, we may define a 1-parameter family of arrangements by \[
  H(t,R)= x\cdot \prod\limits_{k=2}^m(y-m_k x - R\cdot t + (t-t^2)i) \cdot \prod\limits_{j=2}^n (y-w_j x + a_j).
  \] One may see that  \( \arr(0,R)\) is a defining polynomial for  \( \baseHOne \), and that  \( H(t,R) \) is smooth in  \( t. \) Finally, one may see that the family preserves the intersection lattice, thus defines an isotopy of arrangements. Therefore, by  Theorem \ref{diff}, we have that  \( \baseHOne \) and  \( \arr(1,R)\) define arrangements with diffeomorphic complements. 

  Next, we define another two-parameter family of arrangements for  \( t,q \in \mathbb{R} \) by \[
  G(t,q) = x\cdot (y-m_2x - R)) \cdot \prod\limits_{k=3}^m(y-M_k(t,q) x - R) \cdot \prod\limits_{j=2}^n (y-w_j x + a_j).
  \] where  \( M_k(t,q) = \left[\left( \left(\dfrac{q-m_2}{m} k+m_2\right) t  + m_k (1-t)\right) + (t-t^2) i \right] \) for  \( 3 \leq k \leq m. \) One may see that  \( G(0,q) \) is a defining polynomial for  \( \arr(1,R). \)

  Let  \( L_2 \) denote the hyperplane defined by the kernel of   \( \alpha_{L_2} = y-m_2x -R\).
 Let  \( H_1 \) be the hyperplane defined by the kernel of  \( \alpha_{H_1} = x. \) We then have that  \( \baseArr = \{H_1, \cdots, H_n\}. \) Let  \( R^+ = \{(x,y) \in \mathbb{R} | x >0 \}\) be the half-plane with positive  \( x-\)values and let  \( R^+ = \{(x,y) \in \mathbb{R} | x <0 \}\) denote the half-plane with negative  \( x-\)values.

We may partition the set of planes  \( \baseArr \setminus \{H_1\}\) into two sets by defining  \[ \baseHOne^+ = \{ H \in \baseArr  | H \cap L_2 \in R^+ \},  \baseHOne^- =  \{ H \in \baseArr  | H \cap L_2 \in R^- \}. \] Let  \( L(q) \) be the complex line defined by the polynomial  \( y-qx - R.\) Let  \( W \in \mathbb{R} \) be chosen so that  
\begin{itemize}
\item \( W > m_2 \) and
\item  for all  \( W\geq p>m_2 \) we have \( \{  H \in \baseHOne^+ | H \cap L(p) \in R^+\} = \baseHOne^+,\) and \( \{  H \in \baseHOne^- | H \cap L(p) \in R^-\} = \baseHOne^-.\) 
\end{itemize}
Essentially,  \( W \) is chosen so that for all lines of the form  \( y-px-R=0, \) the real part of the line with positive  \( x\)-values intersects the arrangement  \( \baseArr \) in the same set of lines as  \( \baseHOne^+. \)

Returning to the family of arrangements, we now have the 1-parameter family given by  \( G(t,W) \) which is a smooth family of arrangements. We leave to the reader to show that the isotopy preserves the intersection lattice. Therefore,  \( G \) defines a lattice isotopy from the arrangement  \( \arr(1,R) \) defined by  \( G(0,W) \) and the arrangement defined by  \( G(1,W). \) Thus, the arrangement  \( \baseHOne \) is lattice isotopic to the arrangement defined by  \( G(1,W). \)

By re-indexing the set of lines  \( \{H_2, \dots, H_n\} \) we will denote the set of lines in  \( \baseHOne^+=\{H_2, \dots, H_v\} \) for some  \( 2 \leq v \leq n \) where  \( 2 \leq p < q \leq v \) indicates that the distance from the point  \( (0,R) \) to  \( H_p \cap L_2 \) is less than the distance from  \( (0,R) \) to the point  \( H_q \cap L_2. \) Similarly, we re-index so that  \( \baseHOne^- = \{H_{v+1}, \dots, H_n\} \) where again we index by increasing distance from the point  \( (0,R) \) to the point  \( H_p \cap L_2. \) We note that either of these sets may be empty. 

 \begin{figure}[h!]
 \begin{center} 

\includegraphics[width=5in]{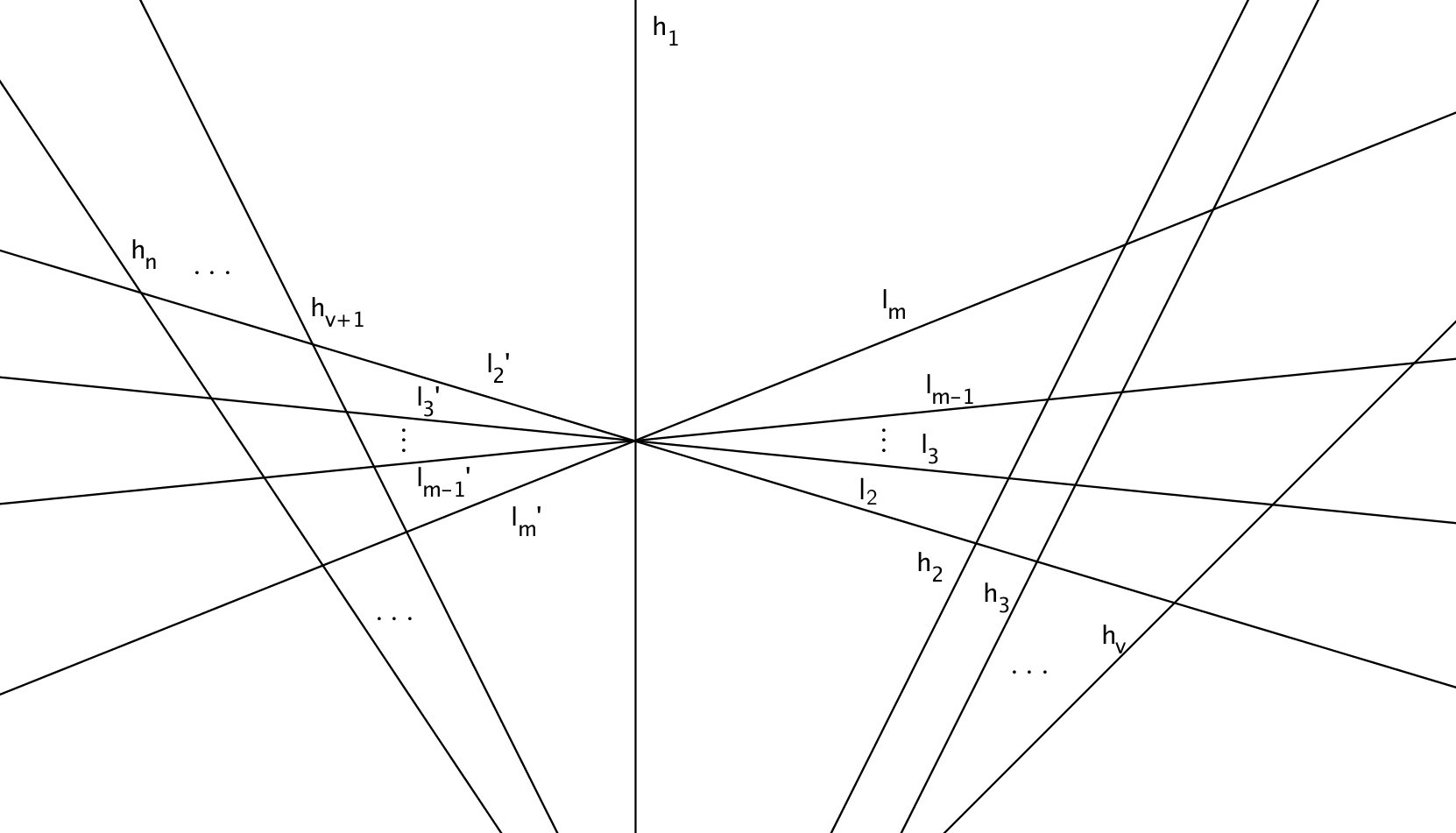}

 		\end{center}
                \caption{The pencil of lines moved ``outside'' of the arrangement  \( \baseArr. \)}
 											\label{pencil}
 					\end{figure}

This reindexing allows us to consider the arrangement of lines pictured in Figure \ref{pencil}. This is the local picture for the arrangement defined by \( G(1,W) \) We denote the lines defined by  \( y-M_k(1,W)x-R \) by \( L_k.\) We may associate to each line  \( L_k \) two generators in the presentation of the fundamental group for the Randell-Arvola presentation. We will denote by  \( l_k \) the generator in  \( R^+ \) and by the  \( l_k' \) the generator in  \( R^-. \) We may also associate to each line  \( H_j \) a generator  \( h_j \)   by choosing the generators to be given by the edge coming from the intersections of  \( H_j \) with  \( L_2. \) Note that since the intersection of \( H_j \) and  \( L_2 \) consists of a point of multiplicity two, by the Randell-Arvola presentation we may use the same generator for both sides of the intersection point.

We will now change the presentation without affecting the homotopy type of resulting 2-complex. (See Section \ref{ssec:complex}). We first note that we have the relations induced by the multiple point at  \( (0,R) \) which give rise to conjugation relators \[
l_2'l_2^{-1},   l_3' (l_3^{l_2})^{-1}, \dots, l_m'(l_m^{l_{m-1}\cdots l_2})^{-1}, h_1'h_1^{-1}
\]
and the commutation relators \[
[h_1,l_m, \dots, l_2]
\]

By choosing the generators in this way, we have the relators along the line  \( L_2 \) consisting of the set \[
V = \{ [l_2, h_k], [l_2',h_j]  \mid H_k \in \baseHOne^+, H_j \in \baseHOne^-\}.
\]
By using the first conjugation relator ( \( l_2'l_2^{-1}\)), we may rewrite the  relators in the set  \( V \) by a Tietze-II move  as \[
\{ [l_2, h_k] \mid 2 \leq k \leq n \}.
\]  

Along the line  \( L_3, \) we have the following relators: \[
\{ [l_3, h_k], [l_3',h_j]  \mid H_k \in \baseHOne^+, H_j \in \baseHOne^-\}.
\] The relator  \(  [l_3',h_j] \) may be rewritten as  \( [l_3^{l_2},h_j] \) by using the conjugation relator  \(l_3' (l_3^{l_2})^{-1}\) and a Tietze-II move. This new relator may be rewritten as  \( A=l_2^{-1}l_3l_2 h_j l_2^{-1}l_3^{-1}l_2 h_j^{-1}.\) Since we have the commutator relator  \( [l_2,h_j]\), the relator  \( A \) may be rewritten as  \( B=l_2^{-1}l_3h_jl_2  l_2^{-1}l_3^{-1} h_j^{-1} \) by a Tietze-III move and may be further rewritten as  \( l_3h_jl_3^{-1}h_j^{-1}=[l_3,h_j] \) via Tietze-I moves. Therefore, the set of relators from the line  \( L_3 \) may be rewritten as \[
\{ [l_3, h_k],  \mid 2 \leq k \leq n\}.
\]
This process may be repeated  for  all \( L_k \) until we have a presentation of \(  \pi_1(M(\baseHOne)) \) given by \[
\left\langle\begin{matrix}h_1,\dots, h_n, \\  l_2,\dots, l_m \end{matrix}  ~\left\lvert~ \begin{matrix} [l_k, h_j]  & \textrm{ for } 2 \leq k \leq m\\ [h_1,l_m,\dots,l_2] & 2 \leq j \leq n \\ R_h & ~\\ \end{matrix} \right. \right\rangle
\]
where  \( R_h \) consists of the relators coming from the arrangement  \( \baseHOne \) induced by the generators  \( h_1, \dots, h_n. \) \end{proof}

%
%
%
%

\begin{lemma}\label{diffeo}

  Let  \( \mathcal{B}_1 \) and  \( \barr_2 \) be  complexified real arrangements in  \( \mathbb{C}^2 \) such that \begin{itemize}
  \item the arrangements  \( \mathcal{B}_i \) are each the union of two subarrangements  \( \baseArr \) and  \( \barr_i' \) that have no lines in common,
  \item  \( \barr_i' \) is a set of  \( m \geq 1 \) parallel lines
  \item every line in  \( \barr_i' \) intersects each line line of  \( \baseArr \) in a point of multiplicity two.
  \end{itemize}
  Then  \( \barr_1 \) and  \( \barr_2 \) have diffeomorphic complements.
\end{lemma}

\begin{proof} We will proceed by finding lattice isotopies that connect these two arrangements. 

  As in the proof of Lemma \ref{present} we may choose coordinates so that the line \( H_1 \in \baseArr  \) is defined by the kernel of  \( f(x,y)=x.\) In this way, we can give a  defining polynomial for  \( \baseArr \) by  \( Q(\baseArr) = x \cdot \prod\limits_{j=2}^n (y-w_j x + a_j) \)  and for  \( \barr_1' \) by  \( Q(\barr_1') = \prod\limits_{k=1}^{m}(y-u_1x-b_k)\) where  \( b_p < b_q \) for  \( p <  q \),  and  \( u_1 \neq w_j \) for all  \( 2\leq j\leq n\). Thus, a defining polynomial for  \( \barr_1 \) is given by  \( Q(\barr_1) = x\cdot \prod\limits_{k=1}^{m}(y-u_1x-b_k)  \cdot \prod\limits_{j=2}^n (y-w_j x + a_j).\)  

  Similarly, we construct a defining polynomial for  \( \barr_2 \) by \( Q(\barr_2) = x\cdot \prod\limits_{k=1}^{m}(y-u_2x-d_k)  \cdot \prod\limits_{j=2}^n (y-w_j x + a_j),\) where  \( d_p  < d_q \) for  \( p<q \), and \( u_2 \neq w_j \) for  \( 2 \leq j \leq n \).

Choose  \( S>0 \) so that all multiple points of  \( \baseArr \) occur in the half-plane defined by  \( y-u_1x-S \leq 0 \) and in the half plane defined by  \( y-u_2 x - S\leq 0. \) Then define the following 1-parameter family of arrangements:  \[
  G_1(t) = x\cdot \prod\limits_{k=1}^{m}(y-u_1x-(b_k(1-t) + (S+k)t + (t-t^2)i)  \cdot \prod\limits_{j=2}^n (y-w_j x + a_j).
  \] We note how this family leaves the subarrangement  \( \baseArr \) unchanged and how  \( G_1(0) \) is a defining polynomial for  \( \barr_1.\) We leave to the reader to verify that  \( G_1 \) is a lattice isotopy.  In the same way, define the lattice isotopy  \[
  G_2(t) = x\cdot \prod\limits_{k=1}^{m}(y-u_2x-(d_k(1-t) + (S+k)t + (t-t^2)i)  \cdot \prod\limits_{j=2}^n (y-w_j x + a_j).
  \] with  \( G_2(0) \) a defining polynomial for  \( \barr_2. \)

  Finally, we define a lattice isotopy that connects the arrangement defined by \( G_1(1) \) to the arrangement defined by  \( G_2(1) \) by  \[
  H(t) = x\cdot \prod\limits_{k=2}^{m}(y-(u_1(1-t) + u_2t + i(t-t^2))x-(S+k)   \cdot \prod\limits_{j=2}^n (y-w_j x + a_j).
  \] Again, as  \( H \) leaves the subarrangement  \( \baseArr \) unchanged, there is little to check to confirm that  \( H \) defines a lattice isotopy of the arrangement defined by  \( G_1(1) \) to the arrangement  \( G_2(1). \)

  Therefore, we have that  \( \barr_1 \) is lattice isotopic to  \( \barr_2 \), therefore \( M(\barr_1) \) is diffeomorphic to  \( M(\barr_2) \).
\end{proof}


%
%
%
%

\subsection{Proof of Theorem  \ref{main} and Theorem \ref{arb}}

\begin{proof} \textit{of Theorem \ref{main}} From Lemma \ref{present}, we have a presentation that generates a 2-complex homotopy equivalent to \(  M(\baseHOne) \) given by 
\[
\left\langle\begin{matrix}h_1,\dots, h_n, \\  l_2,\dots, l_m \end{matrix}  ~\left\lvert~ \begin{matrix} [l_k, h_j]  & \textrm{ for } 2 \leq k \leq m\\ [h_1,l_m,\dots,l_2] & 2 \leq j \leq n \\ R_h & ~\\ \end{matrix} \right. \right\rangle
\]
where  \( R_h \) consists of the relators coming from the arrangement  \( \baseArr \) induced by the generators  \( h_1, \dots, h_n. \)


We perform a final sequence of Tietze transformations to arrive at a preferred presentation. We begin with a Tietze II move by introducing the generator  \( g\) and the relation  \( g=h_1l_m\cdots l_2. \) We then perform another Tietze II move by  deleting the relation  \( gl_2^{-1}\cdots l_m^{-1}=h_1 \) and replace all occurrences of  the generator \( h_1 \) with  \( gl_2^{-1}\cdots l_m^{-1} \). We first examine the consequences on the commutator relations coming from  \(  [h_1,l_m, \dots, l_2].\) Expanding this set of relations and applying the Tietze transformations yields the altered relations:
\begin{align*}
h_1l_m \cdots l_2 &= l_m \cdots l_2 h_1  & \rightarrow & g = l_m \cdots l_2 (gl_2^{-1}\cdots l_m^{-1})   = l_m \cdots l_2 gl_2^{-1}\cdots l_m^{-1} \\
h_1l_m \cdots l_2 &= l_{m-1} \cdots l_2 h_1 l_m & \rightarrow & g = l_{m-1} \cdots l_2 (gl_2^{-1}\cdots l_m^{-1}) l_m = l_{m-1} \cdots l_2 gl_2^{-1}\cdots l_{m-1}^{-1}) \\
\vdots &  \rightarrow & &  \vdots   \\
h_1l_m \cdots l_2  &= l_3l_2h_1l_m \cdots l_4 & \rightarrow & g = l_3 l_2 (gl_2^{-1}\cdots l_m^{-1}) l_m \cdots l_4 = l_3l_2 g l_2^{-1}l_3^{-1} \\
h_1l_m \cdots l_2   &= l_2h_1l_m \cdots l_3 & \rightarrow & g = l_2 (gl_2^{-1}\cdots l_m^{-1}) l_m \cdots l_3 = l_2 g l_2^{-1}\\
\end{align*}

These can then be simplified by freely reducing via Tietze-I moves:

\begin{align*}
  g &= l_m \cdots l_2 (gl_2^{-1}\cdots l_m^{-1})  & \rightarrow & g = l_m \cdots l_2 gl_2^{-1}\cdots l_m^{-1} \\
  g &= l_{m-1} \cdots l_2 (gl_2^{-1}\cdots l_m^{-1}) l_m & \rightarrow & g = l_{m-1} \cdots l_2 gl_2^{-1}\cdots l_{m-1}^{-1} \\
&\vdots &  \rightarrow &   \vdots   \\
  g &= l_3 l_2 (gl_2^{-1}\cdots l_m^{-1}) l_m \cdots l_4 & \rightarrow & g = l_3l_2 g l_2^{-1}l_3^{-1} \\
  g &= l_2 (gl_2^{-1}\cdots l_m^{-1}) l_m \cdots l_3 & \rightarrow & g = l_2 g l_2^{-1}\\
\end{align*}
Working from the bottom of the right column we have the relator  \( [g,l_2] \). Applying a Tietze-III transformation to the second column from the bottom, we may transform the relation  \( g=l_3l_2 g l_2^{-1}l_3^{-1} \) to the relator  \( [g,l_3]. \) Continuing this process will yield the set of relators  \( \left\{ [g,l_k] ~|~ 2 \leq k \leq m\right\}. \) 

Therefore, we have the presentation 
\[
\left\langle\begin{matrix}g,\dots, h_n, \\  l_2,\dots, l_m \end{matrix}  ~\left\lvert~ \begin{matrix} [l_k, h_j]  & \textrm{ for } 2 \leq k \leq m\\ [g,l_k] & 2 \leq j \leq n \\ R_h \left(h_1 \mapsto gl_2^{-1}\cdots l_m^{-1} \right) & ~\\ \end{matrix} \right. \right\rangle
\]
where  \( R_h(h_1 \mapsto gl_2^{-1}\cdots l_m^{-1} ) \) consists of the relators coming from the arrangement  \( \baseArr \) induced by the generators  \( h_1, \dots, h_n \) with  \( h_1 \) replace with  \( gl_2^{-1}\cdots l_m^{-1}. \)  As we have  \( l_k \) commutes with  \( g \) and \( h_j \) for all  \( 2 \leq j \leq n \), we can remove all the  \( l_k \) from the relators in  \( R_h(h_1 \to gl_2^{-1}\cdots l_m^{-1} )  \) and replace with the set of relators  \( R_h(h_1 \to g )  \), that is the set of relators coming from the arrangement  \( \baseArr  \) with  \( h_1 \) replaced by  \( g. \)  Thus we arrive at a presentation we will denote by \(\mathcal{P}_{\baseHOne} \): \[
\mathcal{P}_{\baseHOne} := \left\langle \begin{matrix} g,h_2\dots, h_n, \\  l_2,\dots, l_m \end{matrix} ~\left\lvert ~ \begin{matrix} [l_k, h_j] , &  \textrm{ for } 2 \leq k \leq m,   \\ [g,l_k], &2 \leq j \leq n, \\  R_h(h_1 \mapsto g) & ~ \end{matrix} \right. \right\rangle
\]

Consider the arrangement  \( \mathcal{E} \) defined by  \[ Q(\mathcal{E} ) =   x\cdot (y-m_2x - R)) \cdot \prod\limits_{k=3}^m(y-m_2 x - (R+k)) \cdot \prod\limits_{j=2}^n (y-w_j x + a_j).\] Since  the arrangements \( \mathcal{E} \) and  \( \baseAtInfinity \) satisfy the hypotheses of Lemma \ref{diffeo}, these arrangements are lattice isotopic.  From the value chosen for  \( R \) and as the lines in  \( \prod\limits_{k=3}^m (y-m_2 x - (R+k)) \) in  the polynomial \( Q(\mathcal{E})\) are parallel to  \( L_2 \) we may depict the arrangement as in Figure \ref{parallel}.

 \begin{figure}[h!]
 \begin{center} 

\includegraphics[width=5in]{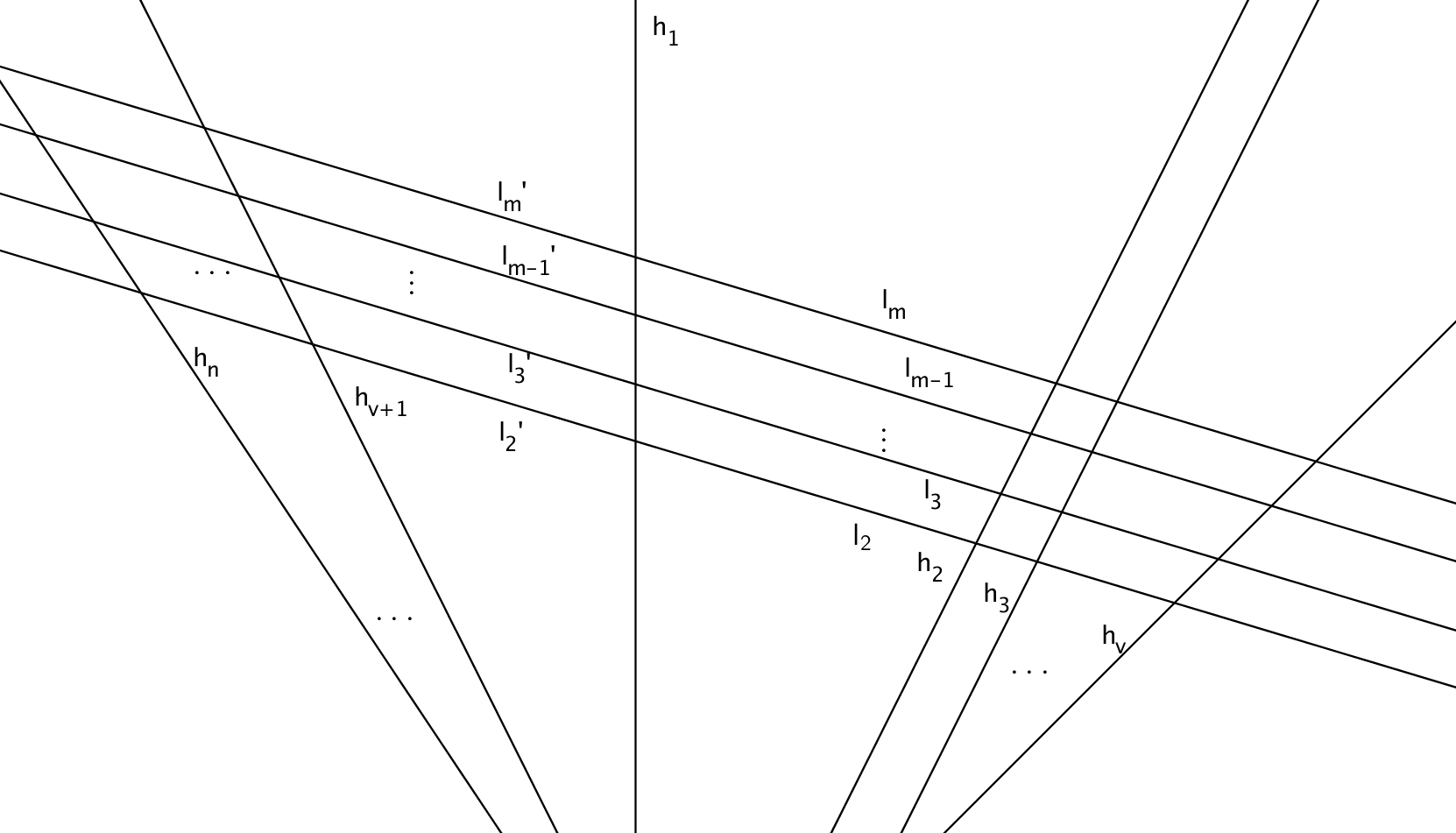}

 		\end{center}
                \caption{The arrangement  \( \mathcal{E} \) consists of  \( \baseArr \) and  a set of parallel lines in general position. }
 											\label{parallel}
 					\end{figure}

In this way, we may label the generators for the fundamental group of \( M(\baseArr) \) by  \( h_1, \dots, h_n \) and label generators for the lines of  \( (y-m_2x-R)\prod\limits_{k=3}^m (y-m_2 x - (R+k)) \) by  \( l_k.\) As all of the  \( h_j's \) and  \( l_k's  \) intersect in points of multiplicity two, we have the relators in the presentation of the fundamental group given by  \(\{ [h_j,l_k]  \mid  1 \leq j \leq n, 2 \leq k \leq m \}.\) The  \( h_j \)'s are the same generators used in the presentation from Lemma \ref{present} so we have a presentation for the fundamental group of the complement of  \( \mathcal{E}\) given by  
\[
\left\langle\begin{matrix}h_1,\dots, h_n, \\  l_2,\dots, l_m \end{matrix}  ~\left\lvert~ \begin{matrix} [l_k, h_j]  & \textrm{ for } 2 \leq k \leq m\\ R_h & 1 \leq j \leq n  \end{matrix} \right. \right\rangle.
\]

Certainly, the 2-complexes built from  \( \mathcal{P}_{\mathcal{E} }  \) and  \( \mathcal{P}_{\baseHOne} \) are homotopy equivalent by relabelling  \( h_1 \) in  \( \mathcal{P}_{\baseAtInfinity} \) by  \( g. \) Therefore we have the following sequence of homotopy equivalencies: \[
M(\baseAtInfinity)  \sim |\mathcal{P}_{\baseAtInfinity} | \sim |\mathcal{P}_{\mathcal{E}}|  \sim |\mathcal{P}_{\baseHOne}| \sim M(\baseHOne).
\]

\end{proof}

Now we use Theorem \ref{main} to prove Theorem \ref{arb} as a corollary.

\begin{proof} \textit{of Theorem \ref{arb}} 
  Let  \( \mathcal{B}_{H_1} \) and \( \mathcal{B}_{H_2} \) be defined as in the statement of the theorem. From Theorem \ref{main}, there exist arrangements \( \baseAtInfinity^1 \) and  \( \baseAtInfinity^2 \) in  \( \mathbb{C}^2 \) such that \begin{itemize}
  \item the arrangement  \( \baseAtInfinity^i \) is the union of two subarrangements  \( \baseArr \) and  \( \parLines \) that have no lines in common,
  \item \( \parLines \) is a set of  \( m-1 \) parallel lines
  \item every line in  \( \parLines \) intersects each line  of  \( \arr \) in a point of multiplicity two.
  \end{itemize} Further, we have that  \( M(\mathcal{B}_{H_i}) \) is homotopy equivalent to  \( M(\baseAtInfinity^i) \).
  From Theorem \ref{diffeo}, we have that  \( \baseAtInfinity^1 \) and  \( \baseAtInfinity^2  \) have homotopy equivalent complements. Therefore we can conclude that  \( M(\mathcal{B}_{H_1})  \) is homotopy equivalent to  \( M(\mathcal{B}_{H_2}) \) for any  \( H_1,H_2 \in \arr. \)
\end{proof}

\section{ Example}\label{examples}

Consider the arrangements  \( \baseHOne \) and  \( \baseAtInfinity \) defined by the following polynomials 
\begin{align*}
Q(\baseHOne) &=  xy (y-1)(y-2)(y+x-2)(y-x) (y+3x+1)(y+4x+1)(y+5x+1)\\
Q(\baseAtInfinity) &= xy(y-1)(y-2)(y+x-2)(y-x)(y+3x+3)(y+3x+2)(y+3x+1)
\end{align*} and depicted in Figure \ref{fig:examples}. If we let  \( \baseArr \) be defined by  \( Q(\baseArr) = xy(y-1)(y-2)(y+x-2)(y-x) \) then we can see that  \( \baseHOne \) and  \( \baseAtInfinity \) define arrangements as described in Theorem \ref{main} with the line  \( H_1 \) defined by the kernel of  \( \alpha_{H_1}=x. \) Therefore,  \( M(\baseHOne) \) is homotopy equivalent to  \( M(\baseAtInfinity). \)

 \begin{figure}[h!]
 \begin{center} 

\includegraphics[width=2in]{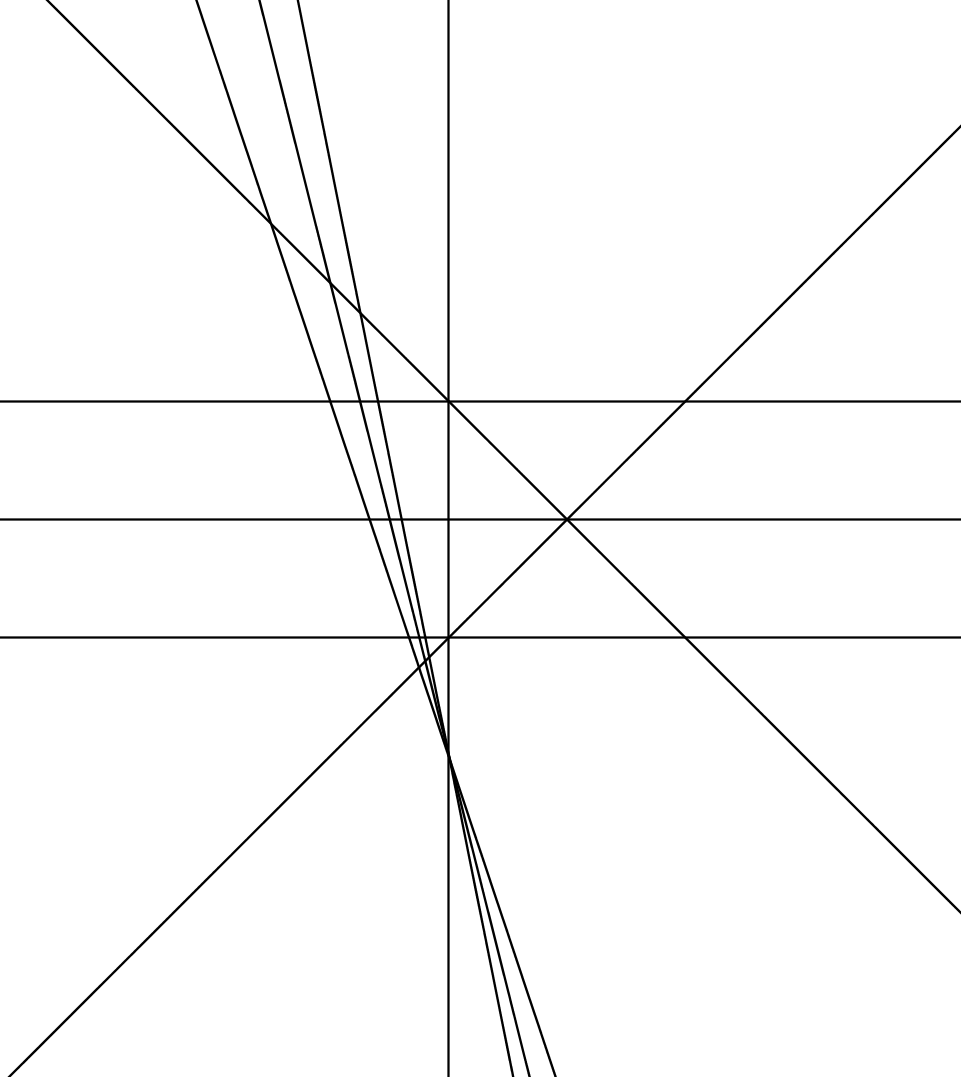} \hspace{0.5in}
\includegraphics[width=2in]{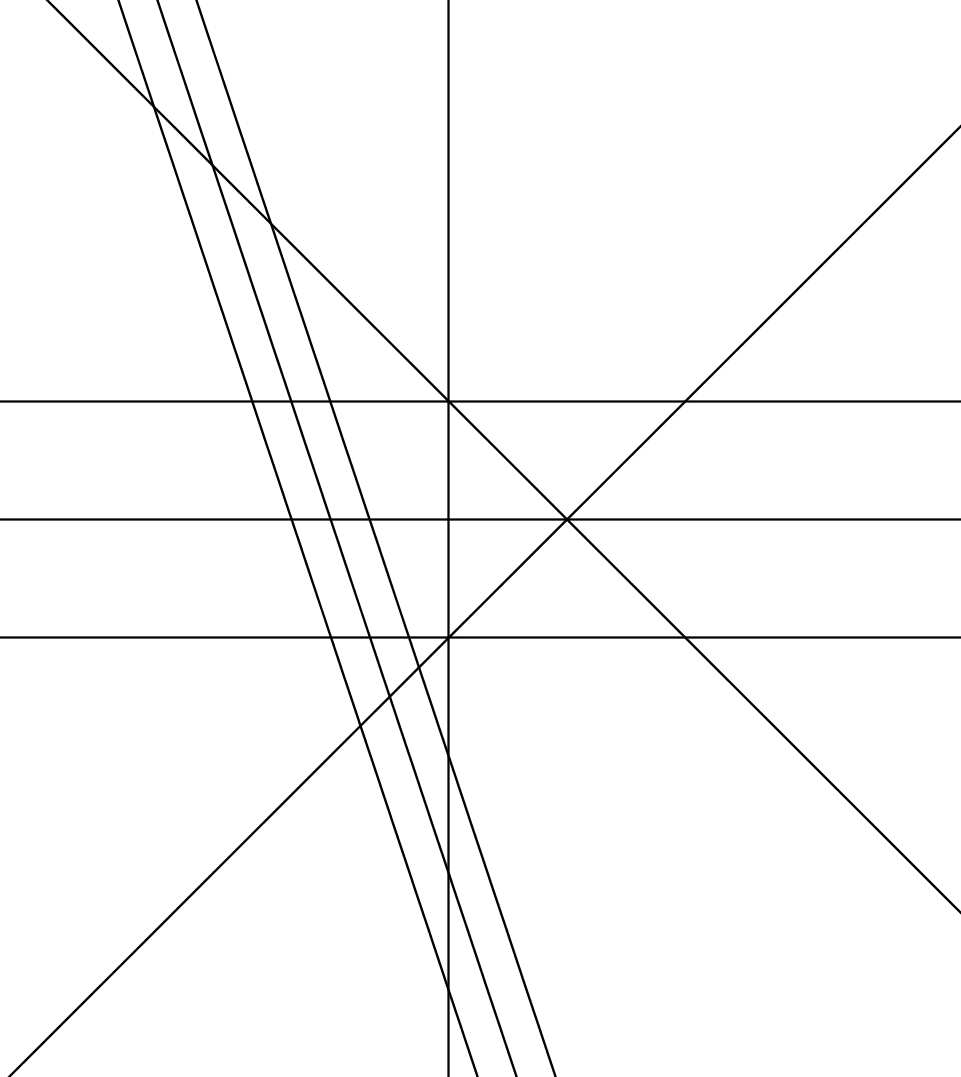}
 		\end{center}
                \caption{The arrangements  \( \baseHOne \) on the left and  \( \baseAtInfinity \) on the right.}
 											\label{fig:examples}
 					\end{figure}

However, we know that  \( M(\baseHOne) \) is not homeomorphic to  \( M(\baseAtInfinity) \). Let  \( \baseHOne^* \) denote the projectivized arrangement in  \( \CP^2 \) formed by embedding  \( \CC^2  \) in  \( \CP^2 \) and adding the ``line at infinity'' to the arrangement  \( \baseHOne. \) 

Then, we can see that  \( M(\baseHOne^*) \) is homeomorphic to  \( M(\baseHOne).\) 

Similarly, we have  \( \baseAtInfinity^* \) as an arrangement in  \( \CP^2 \) with  \( M(\baseAtInfinity^*)  \) homeomorphic to  \( M(\baseAtInfinity). \) As  \( \baseAtInfinity^* \) contains one line with two points of multiplicity four, but  \( \baseHOne^* \) contains no such line, we have that the intersection lattices of  \( \baseAtInfinity^* \) and  \( \baseHOne^* \) are not isomorphic. By a result of \cite{JY-lattice}, this implies that  \( M(\baseAtInfinity^*) \) is not homeomorphic to  \( M(\baseHOne^*). \)

	\bibliographystyle{amsalpha} 


\providecommand{\bysame}{\leavevmode\hbox to3em{\hrulefill}\thinspace}
\providecommand{\MR}{\relax\ifhmode\unskip\space\fi MR }
\providecommand{\MRhref}[2]{%
  \href{http://www.ams.org/mathscinet-getitem?mr=#1}{#2}
}
\providecommand{\href}[2]{#2}


\end{document}